\documentclass[12pt,reqno]{amsart}
\usepackage{amsmath, amsthm, amssymb, stmaryrd}
\usepackage{cleveref}
\usepackage{bbm}
\usepackage{mathrsfs}
\usepackage{makecell}
\usepackage{geometry} 

\usepackage{tikz,ifthen}
\usepackage{graphicx}
\usetikzlibrary{patterns}

\crefformat{section}{\S#2#1#3} 
\crefformat{subsection}{\S#2#1#3}
\crefformat{subsubsection}{\S#2#1#3}

\newtheorem{theorem}{Theorem}[section]
\newtheorem{lemma}[theorem]{Lemma}
\newtheorem{corollary}[theorem]{Corollary}

\theoremstyle{definition}
\newtheorem{definition}[theorem]{Definition}

\theoremstyle{remark}

\numberwithin{equation}{section}

\def\rank{{ \normalfont \text{rank} \,}}

\newcommand{\floor}[1]{\left \lfloor #1 \right \rfloor}
\newcommand{\ceil}[1]{\left \lceil #1 \right \rceil}

\def\bfc{{\mathbf c}}
\def\bfd{{\mathbf d}} \def\bfe{{\mathbf e}}

\def\bfx{{\mathbf x}}

\def\bfZ{{\mathbf Z}}

\def\C{{\mathbb C}}

\def\N{{\mathbb N}}  
\def\R{{\mathbb R}}
\def\Z{{\mathbb Z}}

\def\a{{\alpha}}

\def\eps{\varepsilon}

\def\le{\leqslant} \def\ge{\geqslant}

\begin{document}
\title[multimagic and magic squares]{Existence of $K$-multimagic squares and magic squares of $k$th powers with distinct entries}
\author[Daniel Flores]{Daniel Flores}
\address{Department of Mathematics, Purdue University, 150 N. University Street, West 
Lafayette, IN 47907-2067, USA}
\email{flore205@purdue.edu}
\date{}
\dedicatory{}
\begin{abstract}
 We demonstrate the existence of $K$-multimagic squares of order $N$ consisting of distinct integers whenever $N > 2K(K+1)$. This improves upon our earlier result in which we only required $N+1$ distinct integers. 

Additionally, we present a direct method by which our analysis of the magic square system may be used to show the existence of $N \times N$ magic squares consisting of distinct $k$th powers when
\[N > \begin{cases}
        2^{k+1} & \text{if $2 \le k \le 4$,} \\
        2 \ceil{k(\log k + 4.20032)} & \text{if $k \ge 5$,}
    \end{cases}\]
    improving on a recent result by Rome and Yamagishi.
\end{abstract}
\maketitle

\section{Introduction}\label{sec:intro}

An $N \times N$ matrix $\bfZ = (z_{i,j})_{1 \le i,j \le N}$ is a \emph{magic square} of order $N$ if the sum of the entries in each of its rows, columns, and two main diagonals are equal. Given $K \ge 2$ we say a matrix $\bfZ \in \Z^{N \times N}$ is a $K$-multimagic square of order $N$ or an \textbf{MMS}$(K,N)$ for short if the matrices
\[\bfZ^{\circ k} := (z_{i,j}^k)_{1 \le i,j \le N},\]
remain magic squares for $1 \le k \le K$. Here we expand on our previous investigation \cite{Flores2024}, where we saw that given any $K \ge 2$ and $N \in \N$, there exists many trivial examples of \textbf{MMS}$(K,N)$ utilizing at most $N$ distinct integers. Thus we previously focused on the problem of, given $K$, finding a lower bound for $N$ such that there exists an \textbf{MMS}$(K,N)$ which utilize $N+1$ or more digits. Via the circle method we proved in \cite{Flores2024} that 
\[N>2K(K+1)\] 
is a suitable lower bound on $N$ for this problem.

However, this may not be satisfactory for those familiar with magic squares as the typical parlance usually refers to magic squares with any repeated entries as \emph{trivial}, thus a more satisfactory result would be to determine a lower bound on $N$ in terms of $K$ for which there exists an \textbf{MMS}$(K,N)$ with distinct entries. 

This question has been considered in the past by several authors (see \cite{BoyerSite,Derksen2007,TrumpSite,Zhang2013,Zhang2019}) via constructive methods. We give a brief overview of the best known results in Table \ref{table:kmms}, the curious reader is encouraged to read the introductions of both \cite{Flores2024} and \cite{Rome2024} for more information over the history of magic squares.
\begin{figure}[htbp]
    \centering
    \begin{tabular}{|c|c|c|} 
\hline  
$K$ & \makecell{Smallest $N$ for which an \textbf{MMS}$(K,N)$ \\ with distinct entries is known to exist} & Attributed to \\ 
\hline  
2 & $6$ & J. Wroblewski \cite{BoyerSite} \\
3 & $12$ & W. Trump \cite{TrumpSite} \\
4 & $243$ & P. Fengchu \cite{BoyerSite} \\
5 & $729$ & L. Wen \cite{BoyerSite} \\
6 & $4096$ & P. Fengchu \cite{BoyerSite} \\
$K \ge 2$ & $(4K-2)^K$ & Zhang, Chen, and Li \cite{Zhang2019}\\
\hline  
\end{tabular}
\caption{Best known results for $K$-multimagic squares.}\label{table:kmms}
\end{figure}

Although the circle method tells us there exists \textbf{MMS}$(K,N)$ with at least $N$ distinct entries when $N > 2K(K+1)$, it could be the case that to establish the existence of \textbf{MMS}$(K,N)$ with all distinct entries we may require $N$ to be even larger relative to $2K(K+1)$. This difficulty may be seen in the recent work by Rome and Yamagishi \cite{Rome2024}, where they tackle the simpler problem of showing the existence of magic squares of distinct $k$th powers. Via the techniques in \cite{Rome2024} one may obtain an asymptotic formula for the number of a subset of $N \times N$ magic squares of $k$th powers (with potential repeats) as soon as
\begin{equation*}
    N \ge \begin{cases}
    2^{k+2}+\Delta & \text{if $2 \le k \le 4$,} \\
    4 \ceil{k (\log k + 4.20032)}+\Delta & \text{if $5 \le k$,}
\end{cases}
\end{equation*}
with $\Delta = 12$. However, Rome and Yamagishi end up requiring $\Delta = 20$ to ensure that all entries of these magic squares are \emph{distinct} $k$th powers. To understand why this increase in $N$ is required in \cite{Rome2024}, we first need to establish the notion of a \emph{partitionable matrix}.
\begin{definition}
    We say a matrix $C= [\bfc_1,\bfc_2,\ldots,\bfc_{rn}]$ of dimensions $r \times rn$ is partitionable if there exists disjoint sets $J_l \subset \{1,2,\ldots,rn\}$ of size $r$ for each $1 \le j \le n$ satisfying
    \[\rank (C_{J_l}) = r \text{ for all }1 \le l \le n,\]
    where $C_J$ denotes the submatrix of $C$ consisting of columns indexed by $J$.
\end{definition}
Upon examination of the methods used in \cite{Rome2024}, one sees that $N$ is required to be slightly larger is due to the difficulty of finding a large enough partitionable submatrix for the family of coefficient matrices associated with magic squares with particular repeated entries.

In \cite{Flores2024} we define the notion of a matrix \emph{dominating} a function as follows.
\begin{definition}
    We say that a matrix $C \in \C^{r \times s}$ dominates a function $f:\N \to \R^+$ if for all $J \subset \{1, \ldots, s\}$ we have
\begin{equation*}
    \text{rank}(C_J) \ge \min\left\{f(|J|),r\right\},
\end{equation*}
where $C_J = [\bfc_j]_{j \in J}$.
\end{definition}
Then by \cite[Lemma 1]{Low1988} if a matrix $C$ dominates a certain function we obtain information regarding its partitionable submatrices. This allows us to circumvent several difficulties encountered in \cite{Rome2024} and prove the following result.
\begin{theorem}\label{kmmsquarethm}
    Given $K \ge 2$ there exists infinitely many \textbf{MMS}$(K,N)$ consisting of $N^2$ distinct integers as soon as $N>2K(K+1)$.
\end{theorem}
It is important to note here that our lower bound on $N$ remains unchanged. Additionally, just as in \cite{Flores2024}, one may easily show the following via the Green-Tao theorem.
\begin{corollary}
    Given $K \ge 2$ there exists infinitely many \textbf{MMS}$(K,N)$ consisting of $N^2$ distinct prime numbers as soon as $N>2K(K+1)$.
\end{corollary}

Finally, we present an analogous argument for finding magic squares of distinct k-th powers. This approach utilizes the notion of our matrix of coefficients dominating a particular function, allowing us to establish the following result.
\begin{theorem}\label{kthpowersThm}
    Given $k \ge 2$ there exists infinitely many $N \times N$ magic squares of distinct $k$th powers as soon as
    \[N > \begin{cases}
        2^{k+1} & \text{if $2 \le k \le 4$,} \\
        2 \ceil{k(\log k + 4.20032)} & \text{if $k \ge 5$.}
    \end{cases}\]
\end{theorem}
Thus, improving the recent result of Rome and Yamagishi \cite{Rome2024}. It is worth noting just as Rome and Yamagishi did in \cite{Rome2024pt2} that Theorem \ref{kthpowersThm} is not entirely optimal for $4 \le k \le 20$, we simply choose this representation of our theorem for convenience. We recommend the interested reader to read the introduction of \cite{Rome2024} for more detail on this matter. 

\textbf{Remarks}: The application of the circle method to this problem has been part of the mathematical folklore for at least 30 years, with discussions dating back to the early '90s in talks by Andrew Bremner (see \cite{Bremner1999, Bremner2001}). The recent breakthrough in this area lies in achieving a refined understanding of the coefficient matrix associated with the magic square system. Viewing a matrix as dominating a function appears to be the appropriate perspective, as it provides insight into the partitionability of submatrices. 

\textbf{Acknowledgments}: The author acknowledges the support of Purdue University, which provided funding for this research through the Ross-Lynn Research Scholar Fund. The author expresses profound gratitude to Trevor Wooley for his invaluable mentorship over the past five years and for highlighting Andrew Bremner's talks from the 1990s. Thanks are also extended to Nick Rome and Shuntaro Yamagishi for their insightful and collegial discussions on this problem. Finally, the author is grateful to the referee for their helpful comments.

\section{Finding solutions of additive systems with distinct entries}\label{sec:distinctentries}

Our basic parameter, $P$, is always assumed to be a large positive integer. Whenever $\eps$ appears in a statement, either implicitly or explicitly, we assert that the statement holds for every $\eps>0$. Implicit constants in Vinogradov's notation $\ll$ and $\gg$ may depend on $\varepsilon$, $r$, $s$, $k$, $K$, and the elements of the matrix $C$.

Let $C =[\bfc_1,\ldots,\bfc_s]= (c_{i,j})_{\substack{1 \le i \le r \\ 1 \le j \le s}} \in \Z^{r \times s}$ be given, and consider the diagonal system
\begin{equation}\label{GenSys}
    \sum_{1 \le j \le s}c_{i,j}x_j^k = 0 \quad  1 \le i \le r.
\end{equation}
We define $S_k(P;C)$ to be the set of solutions $\bfx \in \Z^s$ to (\ref{GenSys}) where $\max_{j}|x_j| \le P$. Then given $1 \le i< j \le s$, it is not difficult to see that
\[\#\{\bfx \in S_k(P;C): x_{i} = x_{j}\} = \# S_k\left(P;C^{(i,j)}\right),\]
where $C^{(i,j)}$ is the matrix obtained by substituting the $i$th column of $C$ with $\bfc_{i}+\bfc_{j}$ and deleting the $j$th column of $C$. 

Thus, if $S^*_k(P;C)$ denotes the subset of $S_k(P;C)$ with distinct entries we have that
\begin{equation}\label{splitrepterms}
    \# \bigcap_{1 \le k \le K}S^*_k(P;C) = \# \bigcap_{1 \le k \le K}S_k(P;C) + O\left(\sum_{1 \le i< j \le s} \# \bigcap_{1 \le k \le K}S_k \left(P;C^{(i,j)}\right) \right).
\end{equation}
Separately one may deduce from \cite[Lemma 3.4]{Flores2024} and \cite[Lemma 1]{Low1988} the following result.
\begin{lemma}\label{simplelemma}
    Let $K \ge 2$ and $C \in \Z^{r \times s}$ with $s \ge rK(K+1)$. If $C$ contains a partitionable submatrix of size $r \times rK(K+1)$, then one has the bound
    \[\# \bigcap_{1 \le k \le K}S_k(P;C) \ll P^{s-\frac{rK(K+1)}{2}+\eps}\]
    for any $\eps > 0$.
\end{lemma}

Notice that if $C$ dominates the function
\[\frac{x-r\{(s-2)/r\}}{\floor{(s-2)/2}},\]
then for any $1 \le i <j \le s$ one has that $C^{(i,j)}$ contains a submatrix of $C$ of size $r \times (s-2)$. Then, by \cite[Lemma 1]{Low1988} one has that this matrix contains a partitionable submatrix of size $r \times r\floor{\frac{s-2}{r}}.$ Combining this with (\ref{splitrepterms}), Lemma \ref{simplelemma}, and \cite[Theorem 2.2]{Flores2024} we deduce the following general result.
\begin{lemma}\label{distvinolemma}
    Let $K \ge 2$ and suppose that $C \in \Z^{r \times s}$ satisfies $s \ge rK(K+1)+2$. Then, if $C$ dominates the function
    \[F(x) = \max\left\{ \frac{x- r\{s/r\}}{\floor{s/r}},  \frac{x- r\{(s-1)/r\}}{\floor{(s-1)/r}}, \frac{x- r\{(s-2)/r\}}{\floor{(s-2)/r}} \right\},\]
    we have
    \[\# \bigcap_{1 \le k \le K}S^*_k(P;C) = P^{s-\frac{rK(K+1)}{2}}\left(\sigma_K(C) + o(1)  \right),\]
    where $\sigma_K(C) \ge 0$ is a real number depending only on $K$ and $C$. Additionally $\sigma_K(C) > 0$ if there exists nonsingular real and $p$-adic simultaneous solutions to the system (\ref{GenSys}) for all $1 \le k \le K$.
\end{lemma}

\section{$K$-multimagic squares with distinct entries}

Note that a matrix $\bfZ = (z_{i,j})_{\substack{1 \le i,j \le N}}$ is an \textbf{MMS}$(K,N)$ if and only if for all $1 \le k \le K$ it satisfies the simultaneous conditions
\begin{equation}\label{MSSystem1}
    \sum_{1 \le i \le N} z_{i,j}^{k} = \sum_{1 \le i \le N} z_{i,i}^{k} \quad \text{ for }\quad  1 \le j \le N,
\end{equation}
\begin{equation}\label{MSSystem2}
    \sum_{1 \le j \le N} z_{i,j}^{k} = \sum_{1 \le j \le N} z_{j,N-j+1}^{k} \quad \text{ for }\quad  1 \le i \le N.
\end{equation}
One may wonder if these equations are equivalent to those of an \textbf{MMS}$(K,N)$, indeed it does not seem clear that the main diagonal and anti-diagonal are equal at first glance. One can show that this is implied by the above by simply summing over all $j$ in (\ref{MSSystem1}) and noting that this is equal to summing over all $i$ in (\ref{MSSystem2}). Upon dividing out a factor of $N$ one deduces that (\ref{MSSystem1}) and (\ref{MSSystem2}) imply
\[\sum_{1 \le i \le N} z_{i,i}^{k} = \sum_{1 \le j \le N} z_{j,N-j+1}^{k}.\]
Before we construct a matrix corresponding to this system we must first establish some notational shorthand. Let ${\bf1}_n$, or respectively ${\bf0}_n$, denote a $n$-dimensional vector of all ones, or respectively all zeros. Let $\bfe_n(m)$ denote the $m$th standard basis vector of dimension $n$. For a fixed $N$ we define 
\[D_1(N) = \{(i,j) \in ([1,N] \cap \Z)^2: i = j\},\]
and
\[D_2(N) = \{(i,j) \in ([1,N] \cap \Z)^2: i +j = N+1\}.\]
For each $(i,j) \in ([1,N] \cap \Z)^2$ we define the $2N$-dimensional vectors
\[\bfd_{i,j} = 
\begin{cases}
(\bfe_N(i)-{\bf1_N},\bfe_N(j) - {\bf1_N})  & (i,j) \in D_1(N) \cap D_2(N), \\
(\bfe_N(i)-{\bf1_N},\bfe_N(j))  & (i,j) \in D_1(N) \backslash D_2(N), \\
(\bfe_N(i),\bfe_N(j) - {\bf1_N})  & (i,j) \in D_2(N) \backslash D_1(N), \\
(\bfe_N(i),\bfe_N(j))  & \text{otherwise.} \\
\end{cases}\]
Let $\phi: [1,N^2] \cap \Z \to ([1,N] \cap \Z)^2$ be any fixed bijection, then the $2N \times N^2$ matrix
\[C^{\text{magic}}_{N} = C^{\text{magic}}_{N}(\phi) = [\bfd_{\phi(1)},\ldots,\bfd_{\phi(N^2)}],\]
corresponds to the system defined by (\ref{MSSystem1}) and (\ref{MSSystem2}) up to some arbitrary relabeling of variables defined by the bijection $\phi$. We now establish that one may apply Lemma \ref{distvinolemma} with $C=C_N^{\text{magic}}$.
\begin{lemma}
    For $N \ge 4$, we have that $C_N^{\text{magic}}$ dominates the function
    \[F(x) = \max\left\{ \frac{x- r\{s/r\}}{\floor{s/r}},  \frac{x- r\{(s-1)/r\}}{\floor{(s-1)/r}}, \frac{x- r\{(s-2)/r\}}{\floor{(s-2)/r}} \right\},\]
    with $s = N^2$ and $r=2N$.
\end{lemma}
\begin{proof}
    We show in \cite[Section 4]{Flores2024} that $C_N^{\text{magic}}$ satisfies rank condition
    \[\rank((C_N^{\text{magic}})_J) \ge \begin{cases}
    \ceil{2 \sqrt{|J|}}-1 & \text{if }1 \le |J| \le N(N-1)-1, \\
    |J|-N^2+3N-1 & \text{if }N(N-1)-1 \le |J| \le N(N-1)+1, \\
    2N & \text{if }N(N-1)+1 \le |J| \le N^2.
    \end{cases} \]
    when $|J| \neq (N-1)^2+1$ and
    \[\rank((C_N^{\text{magic}})_J) \ge 2N-3,\]
    when $|J|= (N-1)^2+1$. Given this, if $|J| > N(N-1)$ one trivially has that
    \[\min\{F(|J|),2N\} = 2N \le \rank((C_N^{\text{magic}})_J).\]
    Let us now focus on the case in which $N$ is even and $|J| \le N(N-1)$. When $N$ is even one may show that
    \[F(x) = \begin{cases}
        \frac{2x}{N} & \text{if $0 \le x \le N(N-1)$}, \\
        \frac{2x-4}{N-2}-4 & \text{if $N(N-1) \le x$.}
    \end{cases}\]
    Thus, if $N(N-1)-1 \le |J| \le N(N-1)$ one has
    \[\min\{F(|J|),2N\} = \frac{2|J|}{N} \le |J| - N^2+3N-1 \le \rank((C_N^{\text{magic}})_J).\]
    Let us now suppose that $1 \le |J| < N(N-1)-1$ and $|J| \neq (N-1)^2+1$. Similarly in this range we have that
    \[\min\{F(|J|),2N\}= \frac{2|J|}{N} \le \ceil{2 \sqrt{|J|}}-1 \le \rank((C_N^{\text{magic}})_J).\]
    Let us now consider the case $|J| = (N-1)^2+1$, since $N \ge 4$ we have that
    \[\min\{F(|J|),2N\} = 2N-4+4/N \le 2N-3 \le \rank((C_N^{\text{magic}})_J).\]
    The case in which $N$ is odd may be established in a similar fashion.
\end{proof}
Hence, by Lemma \ref{distvinolemma} and utilizing the existence of nonsingular solutions to the magic square system proved in \cite{Flores2024} we deduce the following asymptotic formula.
\begin{theorem}
    For $K \ge 2$ and $N > 2K(K+1)$, let $M^*_{K,N}(P)$ denote the number of \textbf{MMS}$(K,N)$ consisting of $N^2$ distinct integers bounded in absolute value by $P$. Then there exists a constant $c > 0$ for which one has the asymptotic formula 
    \[M^*_{K,N}(P) \sim cP^{N(N-K(K+1))}.\]
\end{theorem}
This immediately implies Theorem \ref{kmmsquarethm}.

\section{Magic squares of distinct $k$th powers}\label{sec:kthpowers}

Let us now consider the problem of showing the existence of magic squares of distinct $k$th powers. One may repeat the arguments of Rome and Yamagishi \cite{Rome2024} where instead of requiring a lower bound on the size of the largest partitionable submatrix, we assume that the matrix of coefficients dominates the function
\[F(x) = \max\left\{ \frac{x- r\{s/r\}}{\floor{s/r}},  \frac{x- r\{(s-1)/r\}}{\floor{(s-1)/r}}, \frac{x- r\{(s-2)/r\}}{\floor{(s-2)/r}} \right\}.\]
It is clear that all arguments of \cite{Rome2024} follow from this assumption by \cite[Lemma 1]{Low1988} assuming $\floor{\frac{s-2}{r}}$ is large enough in terms of $k$. How large this term needs to be, as seen in Lemma \ref{simplelemma}, is directly determined by when one can establish a mean value estimate of the type
\[\int_{0}^{1} \left|\sum_{x \in A}e(\a x^k) \right|^s d\a \ll (\# A)^{s-k+\eps},\]
where $A$ is the set from which your solutions may come. Then, by a direct analogue of the arguments in Section \ref{sec:distinctentries}, we obtain a version of \cite[Theorem 1.4]{Rome2024} which provides an asymptotic for the number of solutions with distinct entries.
\begin{lemma}\label{distkpowerslemma}
    Let $k \ge 2$ and $C \in \Z^{r \times s}$ where $s \ge r \min\{2^{k},k(k+1)\} +2$. Suppose that $C$ dominates the function
    \[F(x) = \max\left\{ \frac{x- r\{s/r\}}{\floor{s/r}},  \frac{x- r\{(s-1)/r\}}{\floor{(s-1)/r}}, \frac{x- r\{(s-2)/r\}}{\floor{(s-2)/r}} \right\},\]
    then one has that
    \[\#S^*_k(P;C) = P^{s-rk}\left(\sigma_k(C) + o(1)  \right),\]
    where $\sigma_k(C) \ge 0$ is a real number depending only on $k$ and $C$. Additionally $\sigma_k(C) > 0$ if there exists nonsingular real and $p$-adic solutions to the system (\ref{GenSys}).
\end{lemma}
Let
\[\mathscr{A}(Q) = \{\bfx \in \Z^s: \text{prime } p \mid x_i \text{ for any $1 
\le i \le s$, implies } p \le Q\},\]
then the same may be done to obtain an analogue of \cite[Theorem 1.5]{Rome2024} which provides an asymptotic for the number of smooth solutions with distinct entries.
\begin{lemma}\label{distsmoothkpowerslemma}
    Let $k \ge 2$ and $C \in \Z^{r \times s}$ where $s \ge r \ceil{k(\log k + 4.20032)}+2$. Suppose that $C$ dominates the function
    \[F(x) = \max\left\{ \frac{x- r\{s/r\}}{\floor{s/r}},  \frac{x- r\{(s-1)/r\}}{\floor{(s-1)/r}}, \frac{x- r\{(s-2)/r\}}{\floor{(s-2)/r}} \right\},\]
    then when $\eta >0$ is sufficiently small in terms of $s,r,k,$ and $C$ one has that
    \[\# \left(S^*_k(P;C) \cap \mathscr{A}(P^\eta) \right) = c(\eta) P^{s-rk}\left(\sigma_k(C) + o(1)  \right),\]
    where $\sigma_k(C)$ is the same quantity as in Lemma \ref{distkpowerslemma} and $c(\eta)>0$ depends only on $\eta$.
\end{lemma}
Applying Lemma \ref{distkpowerslemma} and Lemma \ref{distsmoothkpowerslemma} with $C = C_N^{\text{magic}}$ and noting that the analysis in \cite[Section 5]{Flores2024} implies $\sigma_k(C_N^{\text{magic}})$ is positive, we deduce Theorem \ref{kthpowersThm}.

\bibliographystyle{amsbracket}
\providecommand{\bysame}{\leavevmode\hbox to3em{\hrulefill}\thinspace}

\end{document}